\documentclass{amsart}

\usepackage[utf8]{inputenc}
\usepackage[OT2,T1]{fontenc}
\DeclareSymbolFont{cyrletters}{OT2}{wncyr}{m}{n}
\DeclareMathSymbol{\Sha}{\mathalpha}{cyrletters}{"58}

\usepackage[hyphens,spaces,obeyspaces]{url}
\usepackage[colorlinks,allcolors=blue,hyperindex,breaklinks]{hyperref}
\hypersetup{
           breaklinks=true,   
           colorlinks=true,   
           pdfusetitle=true,  
        }

\usepackage{orcidlink}

\title[On flat even deformation rings]{On flat even deformation rings}
\date{\today}
\usepackage[foot]{amsaddr}
\author{
 Peter Vang Uttenthal \orcidlink{0009-0001-0878-8213}
 }
\email{petervang@math.au.dk}
\address{Department of Mathematics, Aarhus Universitet, Ny Munkegade 118, 1530-421, DK-8000
Aarhus C, Denmark}
\newcommand{\Gal}{\operatorname{Gal}}

\newcommand{\Q}{\mathbb{Q}}
\newcommand{\Z}{\mathbb{Z}}
\newcommand{\F}{\mathbb{F}}
\newcommand{\W}{\mathbb{W}(\mathbb{F}_{p^f})}

\newcommand{\Ad}{\operatorname{Ad}^0(\overline{\rho})}

\usepackage{amsthm,amsmath, mathrsfs, mathtools}
\usepackage{amsfonts}
\usepackage{amssymb}
\usepackage{fancyhdr}
\usepackage{IEEEtrantools}
\usepackage{tikz-cd}
\usepackage[english]{babel}
\usepackage[utf8]{inputenc}
\usepackage{csquotes}

\newtheorem{theorem}{Theorem}
\newtheorem{lemma}[theorem]{Lemma}
\newtheorem{definition}[theorem]{Definition}
\newtheorem{proposition}[theorem]{Proposition}
\newtheorem{corollary}[theorem]{Corollary}
\newtheorem{remark}[theorem]{Remark}

\begin{document}

\begin{abstract}

In the presence of a nontrivial dual Selmer group, certain global even deformation rings are shown to be finite and flat over $\mathbb{Z}_p$. 
Previously, flatness was only known in established cases of Langlands reciprocity in the odd parity. 
By techniques from global class field theory, explicit examples of even representations are computed to which the results apply. 
For even representations $\overline{\rho}$ in an explicit family, it is observed that if Leopoldt's conjecture is true for a certain number field attached to $\overline{\rho}$, then the global even deformation ring is flat at the minimal level.  
\end{abstract}


\maketitle

\tableofcontents

\section{Introduction}

On the arithmetic side of the Langlands program, the cohomological Selmer and dual Selmer groups 
have been essential objects of study since the work of 
A. Wiles and R. Taylor on the Taniyama-Shimura-Weil conjecture (\cite{Wiles}, \cite{Taylor-Wiles}). 
The Taylor-Wiles method resulted in significant progress on Langlands reciprocity by proving that motivic $L$-functions attached to semistable elliptic curves correspond to automorphic $L$-functions.
Elliptic curves and automophic representations are connected via 2-dimensional Galois representations $\rho$ that are necessarily odd, i.e. complex conjugation $c$ has $\det \rho (c) = -1$. In \cite{Wiles}, the global deformation ring $R$ attached to the odd representation $\rho$ is shown to be isomorphic to a Hecke algebra $T$, which implies that $R$ is finite and flat over the $p$-adic integers $\Z_p$.
For even Galois representations, in the presence of a nontrivial Selmer and dual Selmer group, 
it is not known whether their global deformation rings are finite and flat over $\Z_p$. 
In this paper, 
we construct an infinite family of global even deformation rings $(R_d)_{d\geqslant 2}$ with 
$$
R_d \simeq \Z_p [[T]]/(T^d + a_{d-1}T^{d-1}+\cdots +a_0)
$$
that we prove are finite and flat over $\Z_p$ in global settings where the Selmer and dual Selmer groups are nontrivial. 

To state the results precisely, we introduce some language from the Galois cohomological approach to deformation theory, cf. \cite{artinII} \cite{SFM}. 
Let $\overline{\Q}$ be an algebraic closure of the field $\Q$ of rational numbers, let $\F_p$ be a finite field of cardinality $p$, and let
$$
\overline{\rho}: \Gal(\overline{\Q}) \to \operatorname{GL}(2,\F_p)
$$
be an irreducible residual representation. 
Let $S$ be a finite set of places in $\Q$ containing $p$ and the infinite place such that $\overline{\rho}$ is unramified outside $S$. 
For each place $q$ in $\Q$, let $R_q$ be the local deformation ring at $q$, and let 
$\mathcal{C}_q$ be the set of deformations of $\overline{\rho}|_{\Gal(\overline{\Q}_q/\Q_q)}$ that factor through a fixed smooth quotient 
\[
\begin{tikzcd}
\phi_q: R_q \ar[r, twoheadrightarrow]& \Z_p[[T_1,\ldots,T_{n_q}]]
\end{tikzcd}
\]
for some integer $n_q \geqslant 1$. 
Let $\mathcal{N}_q$ be the subspace of $H^1(\Gal(\overline{\Q}_q/\Q_q),\Ad)$ 
obtained from the monomorphism induced by $\phi_q$ on cotangent spaces. 
One may verify that the pair $(\mathcal{C}_q, \mathcal{N}_q)$
satisfies the properties P1-P7 in \cite[Section 1, p. 551]{artinII}. The collection of subspaces $\mathcal{N} =(\mathcal{N}_q)_{q\in S}$ is called a Selmer condition. 
Let $R_{\mathcal{N}}$ be the global deformation ring parameterizing deformations $\rho$ of $\overline{\rho}$
 with the property that $\rho|_{\Gal(\overline{\Q}_q/\Q_q)} \in \mathcal{C}_q$ for all $q \in S$. The tangent space of $R_\mathcal{N}$ is isomorphic to the Selmer group 
 $H^1_{\mathcal{N}}(\Gal(\Q_S/\Q), \Ad)$
 consisting of all 
 $f\in H^1_{\mathcal{N}}(\Gal(\Q_S/\Q), \Ad)$
 such that $f|_{\Gal(\overline{\Q}_q/\Q_q)} \in \mathcal{N}_q$ for all $q\in S$.

Suppose $\overline{\rho}$ is odd in the sense that $\det \overline{\rho}(c)=-1$ for a complex conjugation $c$. 
In the language of the Galois cohomological approach to deformation theory \cite{artinII} \cite{SFM}, the work of Wiles and Taylor \cite{Wiles} \cite{Taylor-Wiles} has the following consequence:  There exist pairs $(\mathcal{C}_q, \mathcal{N}_q)$ for all $q\in S$ and a Hecke algebra $T$ such that  
$R_{\mathcal{N}} \simeq T$.
In particular, $R_\mathcal{N}$ is finite and flat over $\Z_p$.

If $\overline{\rho}$ is even ($\det \overline{\rho} (c)=1$ for a complex conjugation $c$), then it is not known whether global deformation rings $R_\mathcal{N} = R_{(\mathcal{N}_{q})_{q\in S}}$ 
(unramified away from $S$) are flat over $\Z_p$. It is only known in some special cases, such as when the Selmer group is trivial and $R_{\mathcal{N}} \simeq \Z_p$. 
The purpose of this paper is to construct 
an infinite family of global even deformation rings $(R_d)_{d\geqslant 2}$ of the form  
$$
R_d \simeq \Z_p[[T]]/(T^d+a_{d-1}T^{d-1}+\cdots+a_0)
$$
that are finite and flat over $\Z_p$ even in the presence of a nontrivial Selmer group. 

Let $\Ad^*$ be the dual of $\Ad$,  and let $\mathcal{N}_q^\perp \subseteq H^1(\Gal(\overline{\Q}_q/\Q_q), \Ad^*)$ be the annihilator of $\mathcal{N}_q$.
The dual Selmer group $H^1_{\mathcal{N}^\perp}(\Gal(\Q_S/\Q), \Ad^*)$ is the group of
$\varphi \in H^1(\Gal(\Q_S/\Q), \Ad^*)$ such that $\varphi_{\Gal(\overline{\Q}_q/\Q_q)} \in \mathcal{N}_q^\perp$. 
We say that the global setting is balanced if the Selmer and dual Selmer groups have the same rank (i.e. the same dimension over $\F_p$). 

Let $\overline{\rho}$ be of any parity. 
In \cite{SFM}, it was shown that there is a finite set of primes $Q$ such that after allowing ramification at the primes in $Q$, then 
$$
H^1_{\mathcal{N}}(\Gal(\Q_{S\cup Q}/\Q), \Ad) =
H^1_{\mathcal{N}^\perp}(\Gal(\Q_{S\cup Q}/\Q), \Ad^*) = 0,
$$
from which it follows that
$R_{(\mathcal{N}_q)_{q\in S\cup Q}} \simeq \Z_p$.

In this paper, we will study the situation of an even residual representation in a balanced global setting of rank one (if a global setting is balanced of some positive rank $n$, it can be arranged to be balanced of rank one after allowing ramification at a finite set of primes.) 
In this case, the Weierstrass preparation theorem implies that 
$$
R_\mathcal{N} \simeq \Z_p[[T]]/ p^\mu h
$$
for an integer $\mu$ called the $\mu$-invariant of the ring $R_\mathcal{N}$
and a distinguished polynomial $h(T)$ (recall that a polynomial $h(T) \in \Z_p[T]$ is distinguished if it is monic and the coefficients on the nonleading terms are in the maximal ideal $p\Z_p$).

Suppose $\overline{\rho}: \Gal(\Q_S/\Q) \to \operatorname{GL}(2,\F_p)$ is an irreducible even representation in a balanced global setting of rank one. 
By the methods in \cite{SFM}, there exists a prime $v$ such that after allowing ramification at $v$, the global setting is balanced of rank zero, and 
$$
R_{(\mathcal{N}_q)_{q\in S\cup \{v \}}} \simeq \Z_p.
$$
In particular, the mod $p$ ring has 
$$
\dim R_{(\mathcal{N}_q)_{q\in S\cup \{v \}}}/pR_{(\mathcal{N}_q)_{q\in S\cup \{v \}}} =1.
$$  
\begin{theorem}\label{petervang}
Let $S$ be a finite set of places in $\Q$ containing the prime $p$ and the infinite place.
Let $\Q_S$ be the maximal extension of $\Q$ unramified outside $S$ and consider an irreducible even representation
$$
\overline{\rho}: \Gal(\Q_S/\Q) \to \operatorname{GL}(2,\F_p).
$$
Suppose the global setting is balanced of rank one, in the sense that the Selmer and dual Selmer groups are of equal rank one: 
$$
\dim H^1_{\mathcal{N}}( \Gal(\Q_S,\Q) , \Ad) =
\dim H^1_{\mathcal{N}^\perp}( \Gal(\Q_S,\Q) , \Ad^*) = 1.
$$
Let $d\geqslant 2$ be any integer. 
There exists an infinite set $C^{(p)}_d$ of primes such after after allowing ramification at one additonal prime $v \in C^{(p)}_d$, 
the global setting remains balanced of rank one, 
the global even deformation ring is finite and flat over $\Z_p$ and has the explicit structure 
$$
R_{(\mathcal{N}_q)_{q\in S\cup \{v \} }} \simeq \Z_p[[T]]/(h(T))
$$
where $$h(T)
= T^d + a_{d-1}T^{d-1} + \cdots a_0, \quad a_i \in p \Z_p
$$ is a distinguished polynomial of degree $d$.
In particular, the $\mu$-invariant of the ring $R_\mathcal{N} =R_{(\mathcal{N}_q)_{q\in S\cup \{v \} }}$ is trivial:
$$
\mu(R_{(\mathcal{N}_q)_{q\in S\cup \{v \} }}) = 0.
$$
and the mod $p$ ring is $d$-dimensional over $\F_p$: 
$$
\dim R_\mathcal{N}/pR_\mathcal{N} = d. 
$$
\end{theorem}

\begin{corollary} \label{rankone}
Let $\overline{\rho}$ be an even representation in a balanced global setting of rank one as in Theorem \ref{petervang}.
After allowing ramification at one auxiliary prime in such a way that the global setting remains balanced of rank one, there exists a finite ring extension $\mathcal{O}$ of $\Z_p$
and an even representation 
$$
\rho: \Gal(\Q_{S\cup \{v\}} /\Q) \to \operatorname{GL}(2,\mathcal{O})
$$
with $\rho \equiv \overline{\rho} \mod p$. 
\end{corollary}

\begin{remark}[concerning Theorem \ref{petervang}]
If the global setting is balanced of rank zero at the minimal level, 
the problem of finding auxiliary primes $v$ for which the global setting remains balanced of rank zero after allowing ramification at $v$ is already hard \cite{even2}, and showing that there are infinitely many such primes $v$ is generally an open problem. 
In a forthcoming paper, we study level raising of even $p$-adic representations in balanced global settings of rank zero,
and provide numerical data suggesting that the set of auxiliary primes $v$ for which the 
global setting remains balanced of rank zero after allowing ramification at $v$
has a positive density of \mbox{$(p-1)/p$}.
\end{remark}

\subsection{Related work}
In the Taylor-Wiles method, it is crucial that the representations are of 
\emph{odd} parity, since a Langlands correspondence is established 
between geometric Galois representations attached to semistable elliptic curves and holomorphic modular forms, and these geometric Galois representations are necessarily odd. 
The holomorphic modular forms are parametrized by a Hecke algebra, $T$, 
which is known to be a finite, flat complete intersection, and 
the strategy in \cite{Wiles} is to identify global odd deformation rings $R$ with such Hecke algebras $T$. To show that $R \simeq T,$
it is necessary to bound the rank of the Selmer group at the minimal level (cf. \cite{Wiles} part (i) of Theorem 3.1, p. 518). 
To prove such bounds, the key idea was to \emph{annihilate} the dual Selmer group by allowing ramification at auxiliary primes $q \equiv 1 \mod p$, now called Taylor-Wiles primes (cf. \cite{Wiles} p. 521). This idea was then combined with Wiles' formula (\cite{Wiles} Proposition 1.6) to prove the desired bound on the rank of the Selmer group at the minimal level. 

On the other hand, even Galois representations are not known to come from geometry, 
and they do not correspond to holomorphic automorphic representations, 
so the Taylor-Wiles method is not available to study them.
Langlands reciprocity still contains the even parity as such representations conjecturally correspond to classical Mass wave forms. As F. Calegari explains in \cite{Calegari}, the difficulty is that the Maass forms are incredibly hard to access. This phenomenon accounts for the fact that, compared to the many established cases of odd representations as Langlands parameters,
we have seen almost no progress on reciprocity in the even case. 
Compared to the Taylor-Wiles method, the Galois cohomological method (\cite{SFM}, \cite{Patrikis1}) provides an alternative lifting technique founded purely on
Galois cohomology, which makes it a suitable method for 
parity-free global settings. The method has been significantly extended by N. Fakhruddin, C. Khare and S. Patrikis in \cite{FKP1} and \cite{FKP}.

In the Galois cohomological method, the idea is to create a global setting which is balanced, in the sense that the Selmer and dual Selmer groups are of equal rank. Once this has been achieved, Taylor and Wiles' idea of modifying the rank of the dual Selmer group with auxiliary primes is brought in. 
The auxiliary primes $q$ should have a balanced local deformation theory in such a way that the global setting remains balanced after allowing ramification at $q$.
If the auxiliary primes are chosen to lower the Selmer rank while keeping the global setting balanced, then the dual Selmer rank is annihilated with a finite number of auxiliary primes. At this point, the global setting is balanced of rank zero, and when this happens, it was observed in \cite{even1} (and generalized significantly in \cite{lifting}, \cite{SFM}, \cite{artinII}) that there is an inductive method of infinitesimally deforming the representation to characteristic zero. 
Since the method proceeds by creating a balanced global setting of rank zero, an independent motivation for this work was to loosen this criteria and prove existence of characteristic zero lifts in balanced global settings of positive rank. In Theorem \ref{rankone}, we achieve this in balanced global settings of rank one. We expect that it is possible to generalize the result from rank one to balanced global settings of arbitrary rank, but we are still searching for a way to move beyond rank two. 

In \cite{KRtorsion} concerning torsion Galois representations, 
certain applications of Wiles' formula in the proof of Theorem 2 (\cite{KRtorsion} p. 28-30) turned out to extend cleanly to the parity-free setting of this paper, although the authors had a different focus than us. A ring of integers $\mathcal{O}$ in a finite extension of $\Q_p$ with uniformizer $\pi$ was fixed, and then odd, irreducible representations with coefficients in the torsion module $\mathcal{O}/(\pi^n)$ were studied. The authors proved that by allowing ramification at a finite set of auxiliary primes, the ring $\mathcal{O}$ 
can be realized as the global odd deformation ring parametrizing newforms of weight $k=2$. 
The work deals exclusively with odd representations in residual characteristic 
$p\geq 5$, where they work with the type of auxiliary primes called \emph{nice} primes (see Definition \ref{nice} below). 
In this paper, we allow $p \geq 3$. For $p=3$, where nice primes are not defined, we work with a different type of auxiliary primes whose local deformation theory comes with cohomology groups of exactly the ranks we need. Our setting is parity free, in the sense that we do not need to assume either oddness or evenness of the representations in our proofs.

\subsection{Summary}
The contents of the following sections are briefly summarized.
In Theorem \ref{rankone}, certain global even deformation rings in characteristic $p\geqslant 3$ are shown to be finite and flat over $\mathbb{Z}_p$ even without annihilating their dual Selmer groups. 
Explicit examples of even representations are given to which Theorem \ref{rankone} applies. Either the associated global even deformation rings are flat at the outset, or we can allow ramification at a single auxiliary prime to make the new part flat. In either case, the global setting remains balanced with a nontrivial dual Selmer group, and despite this obstruction, we prove existence of even characteristic zero lifts. 
For even representations $\overline{\rho}$ in an explicit family, we show that there is a number field $K_1$ attached to $\overline{\rho}$ such that if Leopoldt's conjecture holds for $K_1$, then the global even deformation ring is flat at the minimal level.  

\subsection{Acknowledgments}
I am grateful to Ravi Ramakrishna and Gebhard B\"{o}ckle for sharing their ideas, to the referee for their many valuable suggestions, and to Paul Nelson for helpful conversations. This work was supported by the grant VIL54509 from Villum Fonden.

\section{The Galois cohomological method for reductive groups}
The purpose of this section is to formulate the notions of the introduction in precise terms, and to present some background for the convenience of the reader. Since its inception \cite{SFM}, 
Patrikis \cite{Patrikis1} and Fakhruddin-Khare-Patrikis \cite{FKP1} have extended the Galois cohomological method extensively to include
Galois representations
$$
\overline{\rho}: \Gal(\overline{\Q}/\Q) \to G(\overline{\F}_p)
$$
valued in general reductive groups $G$, and the currect section is written in the modern language. 
The extension of the method is crucial for  applications in the Langlands program, whose purpose is to relate harmonic analysis on various reductive groups to arithmetic. 
In later sections, we will specialize to $G=\operatorname{GL}(2)$ or 
$G=\operatorname{SL}(2).$ In future work, we hope to extend Theorem $\ref{rankone}$ to general reductive groups. 

For a prime number $p$, let $S$ be a finite set of places $v$ in $\Q$ containing the wild place $v=p$ and the archimedian place $v=\infty$, and let $\Q_S$ be the maximal extension of $\Q$ unramified outside $S$. By requiring $S$ to contain $\infty$, we make our setting conducive to applications of Wiles' formula, which rests on global Tate duality and the global Euler characteristic formula. 
Let $\overline{\Z}_p$ be the integral closure of ${\Z}_p$ in $\overline{\Q}_p$.
Let $G$ be a smooth group scheme over $\overline{\Z}_p$ for which the irreducible component containing the identity, $G^0$, is a split connected reductive group, and consider a residual representation
$$\overline{\rho}: \Gal(\Q_S/\Q) \to G(\overline{\F_p}).$$ The theory seeks, as a basic goal, to prove existence of representations $\rho$ to charactistic zero such that the diagram
\[ 
\begin{tikzcd}
    & G(\overline{\Z}_p) \arrow[d]\\ 
  \Gal(\Q_S/\Q) \arrow[r, "\overline{\rho}"]   \arrow[ur, dotted, "\rho" description]  & G(\overline{\F}_p)
\end{tikzcd}
\]
commutes (\cite{FKP1}, p. 3506). Mazur had the idea of regarding this problem from the point of view of geometry. In \cite{Mazur1}, he proved that there is a type of moduli space, $X^{\operatorname{ver}}= \operatorname{Spec} R^{\operatorname{ver}}$, which is a $\Q_p$-analytic manifold parameterizing all lifts
of $\overline{\rho}$ up to strict equivalence. From this point of view, characteristic zero lifts correspond to $\overline{\Z}_p$-valued points of the space $X^{\operatorname{ver}}$.  
In Mazur's theory,
the tangent space can be interpreted as the degree one cohomology group 
$H^1(\Gal(\Q_S/\Q), \operatorname{Ad}(\overline{\rho}))$, where $\operatorname{Ad}(\overline{\rho})$ is the $\F_p[\operatorname{Gal}(\Q_S/\Q)]$-module obtained by composing $\overline{\rho}$ with the adjoint representation of the Lie algebra of $G$.
This opens up the possibility of bringing in Galois cohomology as a major tool, and the Galois cohomological method can be seen as an extension of Mazur's work that pushes this perspective forward, by showing that one can construct a deformation theory of Galois representations solely on a foundation of Galois cohomology, and then prove many of the theorems that used to require a range of automorphic and geometric techniques.
The method rests on a local-global principle:
Locally at places away from the wild prime $p$, 
$G(\overline{\F}_p)$-valued Galois representations factor through quotients of the local absolute Galois groups with simple presentations in terms of generators and relations, and therefore such local lifting problems are easier to solve. More precisely, at unramified places $v$, the group $\Gal(\Q^{\operatorname{unr}}_v/\Q_v) \simeq \hat{\Z}$ has cohomological dimension one, and local characteristic zero lifts at $v$ automatically exist. At tamely ramified places $v$, the group $\Gal(\Q^{\operatorname{tame}}_v/\Q_v)$ is topologically generated by inertia $\tau_v$ and a lift of the Frobenius automorphism, $\sigma_v$, subject to the single relation $\sigma_v\tau_v \sigma_v^{-1} = \tau_v^v.$
In the Galois cohomological method, the local-global principle 
is summarized by the observation that obstructions to the global lifting problem is detected by the cohomological Selmer group.
\begin{definition}
    Let $\Gamma$ be a topological group, $A$ a ring, $H$ an affine group scheme over $\operatorname{Spec}A$, and 
$\rho: \Gamma \to H(A)$ a continuous homomorphism. If the vector space $V$ is a representation of $H$, we let $\rho(V)$ be the $A[\Gamma]$-module
induced by $\rho$. 
\end{definition}

For the reductive group $G$, let $\mathfrak{g}$ denote the Lie algebra of $G$.
For each $v\in S$, let $\Gamma_v= \operatorname{Gal}(\overline{\Q}_v/\Q_v).$
Let 
$\mathcal{N}_v$ be a subspace of $H^1(\Gamma_v, \overline{\rho}(\mathfrak{g}))$. The collection of subspaces $\mathcal{N} := (\mathcal{N}_v )_{v\in S}$ is called a Selmer condition. 
Let $\Gamma_S = \operatorname{Gal}(\Q_S/\Q).$  

\begin{definition}
The Selmer group
$H^1_{\mathcal{N}} (\Gamma_S, \overline{\rho}(\mathfrak{g}))
$
is defined to be the kernel of the map
$$
H^1 (\Gamma_S, \overline{\rho}(\mathfrak{g})) \to \bigoplus_{v\in S} H^1(\Gamma_v, \overline{\rho}(\mathfrak{g})) / \mathcal{N}_v.
$$
Let $\mu_p$ be the group scheme of $p^{th}$ roots of unity and 
let $\overline{\rho}(\mathfrak{g})^* = 
\operatorname{Hom}(\overline{\rho}(\mathfrak{g}) ,\mu_p)$ denote the dual module of $\overline{\rho}(\mathfrak{g})$.
For each $v\in S$, let $\mathcal{N}_v^\perp
\subseteq H^1(\Gamma_v, \overline{\rho}(\mathfrak{g})^*)$ be the annihilator of $\mathcal{N}_v$ under the local pairing of Galois cohomology. The dual Selmer group 
$H^1_{\mathcal{N}^\perp} (\Gamma_S, \overline{\rho}(\mathfrak{g})^*)$
is the kernel of the map
$$
H^1 (\Gamma_S, \overline{\rho}(\mathfrak{g})^*) \to \bigoplus_{v\in S} H^1(\Gamma_v, \overline{\rho}(\mathfrak{g})^*) / \mathcal{N}_v^\perp.
$$
\end{definition}
For $G=GL(2),$ the module $\overline{\rho}(\mathfrak{g})$ (usually denoted
$\operatorname{Ad}(\overline{\rho})$) is the space of 2-by-2 matrices over the relevant finite field, 
and the Galois action induced by 
$\overline{\rho}$ is via composition with matrix conjugation. 

Let $f \geq 1$ be an integer such that 
$\overline{\rho}$ has image in $G(\F_{p^f})$.
Denote by $\W$ the ring of Witt vectors over $\F_{p^f}$.
For each $v\in S$, suppose the local deformation ring $R_v$ has a smooth quotient
$$
R_v \rightarrow \W[[T_1,\ldots,T_{n_v}]].
$$
The dual morphism on tangent spaces defines a subspace 
$\mathcal{N}_v \subseteq H^1(\Gamma_v, \overline{\rho}(\mathfrak{g}) )$.
Let $\mathcal{C}_v$ be the set of deformations of $\overline{\rho}|_{\Gamma_v}$ that factor through the smooth quotient.
Then $\mathcal{N}_v$ is the space of local cohomology classes that preserve $\mathcal{C}_v,$ and $n_v=\dim \mathcal{N}_v.$
The Selmer condition 
$\mathcal{N}=
(\mathcal{N})_{v \in S}$
gives rise to deformation conditions that define a global deformation ring
$R_{(\mathcal{N})_{v \in S}}$ parametrizing all deformations of $\overline{\rho}$
that locally factor through the smooth quotients
$$
R_v \rightarrow \W[[T_1,\ldots,
T_{\dim \mathcal{N}_v} ]]
$$
for all $v\in S$. 

\begin{definition} \label{balancedrankn}
    Let 
    $$
\overline{\rho}: \Gamma_S \to G( \F_{p^f}),
    $$
be a residual Galois representation.
We say that the global setting is balanced of rank $n$ if
$$n=\dim_{\F_{p^f}} H_{\mathcal{N}}^1(\Gamma_S, \overline{\rho}(\mathfrak{g})) 
    = \dim_{\F_{p^f}} H_{\mathcal{N}^\perp}^1(\Gamma_S,\overline{\rho}(\mathfrak{g})^*).$$
\end{definition}

To create global settings that are balanced, we will invoke Wiles' formula, which controls the relative sizes of the Selmer group and the dual Selmer group. 
\begin{proposition}[Wiles' formula]
    Let $T$ be a finite set of places containing $S$. 
    Let $M$  be a finite-dimensional vector space over a finite field equipped with a Galois action. 
    For each $v\in S,$ let $\mathcal{L}_v$ be a subspace of $H^1(\Gamma_v, M)$ with annihilator $\mathcal{L}_v^\perp \subseteq H^1(\Gamma_v, M^*)$ under the local pairing. Then 
\begin{IEEEeqnarray*}{ll}
 & \dim H^1_\mathcal{L} (\Gamma_T, M) - \dim H^1_{\mathcal{L}^\perp} (\Gamma_T, M^*) \\
= \quad & 
\dim H^0 (\Gamma_T, M) - \dim H^0 (\Gamma_T, M^*) 
+
\sum_{v\in T} \dim \mathcal{L}_v - \dim H^0(\Gamma_v, M).  
\end{IEEEeqnarray*}
\end{proposition}
\begin{proof}
    See Proposition 1.6 in \cite{Wiles}.
\end{proof}

\section{Balanced global settings of rank one}
For a finite set $S$ of places of $\Q$ containing $p$ and the archimedian place,
let
$\Gamma_S=\operatorname{Gal}(\Q_S/\Q)$. 
In this section, we consider irreducible Galois representations
$$
\overline{\rho}: \Gamma_S \to GL(2,\F_{p^f}). 
$$
with image containing  $SL(2, \F_{p^f})$.
We fix determinants, while we do not impose any parity condition on $\overline{\rho}$. With applications of Wiles' formula in mind, we will require that the spaces of global Galois invariants in the adjoint representation and its dual are of the same rank:
$$
\dim H^0(\Gamma_S, \operatorname{Ad}^0(\overline{\rho})) = \dim H^0(\Gamma_S, \operatorname{Ad}^0(\overline{\rho})^*).
$$
Let $\W$ be the ring of Witt vectors over $\F_{p^f}$.
\begin{lemma} \label{flatness}
Let $(\mathcal{N}_v)_{v \in S}$ be a Selmer condition. 
If the global deformation ring $R_{(\mathcal{N}_v)_{v \in S}}$ is flat over $\W$, then there exists a finite ring extension $\mathcal{O}$ of $\Z_p$ and
        a Galois representation 
$$
\rho: \Gamma_{S} \to GL(2, \mathcal{O})
$$
lifting $\overline{\rho}$ such that $\rho|_{G_v} \in \mathcal{C}_v$ for all $v\in S$.
\end{lemma}

\begin{proof}
This Lemma follows from the work of Khare and Wintenberger on Serre's conjecture (\cite{KW2} Proposition 2.2(i) p. 511).
\end{proof}

\begin{remark}[Remark to Lemma \ref{flatness}]
The ring $\Z_p[[T]]/(pT)$ is not finite and flat over $\Z_p$, while it has a characteristic zero point (obtained by mapping $T$ to zero). 
\end{remark}
By $\operatorname{Ad}^0(\overline{\rho}),$ we mean the Galois stable subspace
of the adjoint representation consisting of traceless matrices. 
\begin{lemma} \label{1gen1rel}
Consider a residual Galois representation    
$$
\overline{\rho}: \Gamma_S \to GL(2, \F_{p^f}).
$$
If the global setting is balanced of rank one
in the sense that 
$$1=\dim_{\F_{p^f}} H_{\mathcal{N}}^1(\Gamma_S, 
\operatorname{Ad}^0(
\overline{\rho})) 
    = \dim_{\F_{p^f}} H_{\mathcal{N}^\perp}^1(\Gamma_S,
    \operatorname{Ad}^0(
\overline{\rho}))^*),$$
then
$$
R_{(\mathcal{N}_v)_{v \in S}} \simeq \W[[T]]/(f)
$$
for some $f \in \W[[T]].$
\end{lemma}
\begin{proof}
    The global deformation ring $R_{(\mathcal{N}_v)_{v \in S}}$ has
    $\dim H_{\mathcal{N}}^1(G_S, Ad^0(\overline{\rho})) = 1$ 
     generators, and 
     $ \dim H_{\mathcal{N}^\perp}^1(G_S,Ad^0(\overline{\rho})^*) = 1$ 
     relations (see the proof of \cite[Lemma 1.1]{artinII} and the remarks following the proof). 
\end{proof}

\subsection{Auxiliary primes for $p\geq 5$}

First, we focus on residual representations
in characteristic $p\geq 5$ and identify a suitable type of auxiliary primes in this case. 

\begin{definition}[Nice primes] \label{nice}
Let $p\geq 5$ and consider an irreducible
$$
\overline{\rho}: \Gamma_S \to GL(2, \F_{p^f}).
$$
with image containing  $SL(2, \F_{p^f}).$
    A prime $v$ is nice for $\overline{\rho}$ if 
    \begin{enumerate}
        \item $v \not \equiv \pm1 \mod p$
        \item $\overline{\rho}$ is unramified at $v$
        \item $\overline{\rho}(\sigma_v)$ has eigenvalues of ratio $v$, where $\sigma_v$ denotes the Frobenius automorphism at $v$. 
    \end{enumerate}
\end{definition}

\begin{lemma} \label{localranks5} Let $p\geq 5$ and let $v$ be a nice prime. 
Suppose $\overline{\rho}$ is irreducible with image containing
$\operatorname{SL}(2,\F_{p^f}).$
The local cohomology groups at $v$ have the properties that
$$
\dim H^0(\Gamma_v, \operatorname{Ad}^0(\overline{\rho})) = 1,
\quad 
\dim H^2( \Gamma_v, \operatorname{Ad}^0(\overline{\rho})) = 1,
\quad 
\dim H^1(\Gamma_v, \operatorname{Ad}^0(\overline{\rho})) = 2.
$$
\end{lemma}
\begin{proof} See \cite{KLR1} Lemma 2 p. 714. 
\end{proof}
Let $p\geq 5$, let $v$ be a nice prime, 
and let $\sigma_v, \tau_v$ be topological generators of the tame quotient of $\Gamma_v$, where $\sigma_v$ is a lift of Frobenius and $\tau_v$ is a generator of inertia. 
Up to a twist (cf. \cite{KLR1} p. 714),
$$
\overline{\rho}|_{G_v}: 
\sigma_v \mapsto 
\begin{pmatrix}
v & 0 \\
0 & 1
\end{pmatrix},
\quad 
 \tau_v \mapsto 
\begin{pmatrix}
1 & 0 \\
0 & 1
\end{pmatrix}.
$$
Let $\W$ be the ring of Witt vectors over $\F_{p^f}$.
Let $\mathcal{C}_v$ be the set of deformations of 
$\overline{\rho}|_{\Gamma_v}$ valued in $\operatorname{GL}(2,A)$ for Artin algebras $A$ over $\W$
such that 
$$
\pi_v: 
\sigma_v \mapsto 
\begin{pmatrix}
v & 0 \\
0 & 1
\end{pmatrix},
\quad 
 \tau_v \mapsto 
\begin{pmatrix}
1 & y \\
0 & 1
\end{pmatrix},
$$
where $y$ is in the maximal ideal of $A$. 
Let $\mathcal{N}_v$ be the maximal subspace of $H^1(\Gamma_v\Ad)$ such that for any small extension of Artin algebras $A\to B$,
a generator $\varepsilon \in \ker(A\to B)$, and any $\rho_v \in \mathcal{C}_v$ with coefficients in $B$, we have 
\[
(f\in \mathcal{N}_v) \implies (1+\varepsilon f)\rho_v \in \mathcal{C}_v.
\]
The space $\mathcal{N}_v$ is called the subspace of 
$H^1(\Gamma_v, \operatorname{Ad}^0(\overline{\rho}))$ that preserves $\mathcal{C}_v$, and
the pair $(\mathcal{C}_v, \mathcal{N}_v)$ satisfies the properties P1-P7 in \cite{artinII}. 
\begin{lemma} \label{vnice}
Let $p \geq 5$ and let $v$ be a nice prime. 
Suppose $\overline{\rho}$ is irreducible with image containing
$\operatorname{SL}(2,\F_{p^f}).$
Let  $\mathcal{N}_v$ be the subspace of $ H^1(\Gamma_v, Ad^0(\overline{\rho}))$ that preserves $\mathcal{C}_v$.
The pair $(\mathcal{N}_v, \mathcal{C}_v)$ at a nice prime $v$ is balanced in the sense that  
$$
\dim \mathcal{N}_v = \dim H^0(\Gamma_v, \operatorname{Ad}^0(\overline{\rho}) )
$$
and is spanned by the local cohomology class
\begin{equation}\label{act}
r_v: 
\sigma_v \mapsto 
\begin{pmatrix}
0 & 0 \\
0 & 0
\end{pmatrix},
\quad 
 \tau_v \mapsto 
\begin{pmatrix}
0 & 1 \\
0 & 0
\end{pmatrix}.
\end{equation}
\end{lemma}
\begin{proof} With respect to the action described in \eqref{act},
the cohomology class $r_v$ preserves $\mathcal{C}_v$, yet not every class in $H^1(\Gamma_v, \Ad)$ has this property.
By Lemma \ref{localranks5}, $\dim H^1(\Gamma_v \Ad) = 2$, hence
$\dim \mathcal{N}_v = 1$ and $\operatorname{span}_{\F_p}\{r_v \} = \mathcal{N}_v$.
Once we notice that $\dim H^0(\Gamma_v, \Ad) =1$ (cf. Lemma \ref{localranks5}), the lemma follows. 
\end{proof}
\begin{lemma} \label{remainbalanced5}
Let $p\geq 5.$
    Suppose 
    $$
    \overline{\rho}: \Gamma_S \to GL(2, \F_{p^f})
    $$
    is irreducible with image containing
$\operatorname{SL}(2,\F_{p^f})$ and that the global setting is balanced. 
    If we allow ramification at an auxiliary nice prime $v$, then the global setting remains balanced in the sense that
    $$
\dim H^1_{\mathcal{N}} (\Gamma_{S\cup\{v\}}, \operatorname{Ad}^0(\overline{\rho}))
=\dim H^1_{\mathcal{N}^\perp} (\Gamma_{S\cup\{v\}}, \operatorname{Ad}^0(\overline{\rho})^*).$$
\end{lemma}
\begin{proof}
    This follows from Lemma \ref{vnice} and Wiles' formula. 
\end{proof}

\subsection{Auxiliary primes for $p=3$}
Next, we turn our attention to the case of residual representations
in characteristic $p=3$ 
and identity the suitable class of auxiliary primes in this case. 
Again, we assume that $\overline{\rho}$ is irreducible with image containing
$\operatorname{SL}(2,\F_{3^f}).$
\begin{definition} \label{C3}
    Let $C^{(3)}$ be the Chebotarev set of primes $v$ such that 
\begin{enumerate}
    \item $v \equiv 1 \mod 3,$ 
    \item The Frobenius automorphism at $v$ has order 3 in $\Q(\overline{\rho})/\Q$.
\end{enumerate}
\end{definition}

Let $v \in C^{(3)}$ and fix a basis such that
$$
\overline{\rho}|_{G_v}: 
\sigma_v \mapsto 
\begin{pmatrix}
1 & 1 \\
0 & 1
\end{pmatrix},
\quad 
 \tau_v \mapsto 
\begin{pmatrix}
1 & 0 \\
0 & 1
\end{pmatrix}.
$$
\begin{lemma} \label{localranks3}
The local cohomology groups at $v$ have the properties that
$$
\dim H^0(\Gamma_v, Ad^0(\overline{\rho})) = 1,
\quad 
\dim H^2(\Gamma_v, Ad^0(\overline{\rho})) = 1,
\quad 
\dim H^1(\Gamma_v, Ad^0(\overline{\rho})) = 2.
$$
\end{lemma}
Let $\mathcal{C}_v$ be the set of deformations of 
$\overline{\rho}|_{G_v}$ valued in $\operatorname{GL}(2,A)$ for Artin algebras $A$ over $ \mathbb{W}(\F_{3^f})$
such that 
\begin{equation}\label{good}
\pi_v: 
\sigma_v \mapsto 
\begin{pmatrix}
\sqrt{v} & 1 + x \\
0 & \sqrt{v}^{-1}
\end{pmatrix},
\quad 
 \tau_v \mapsto 
\begin{pmatrix}
1 & y \\
0 & 1
\end{pmatrix}
\end{equation}
for some $x$ and $y$ in the maximal ideal of $A$. 
Note that $v\equiv 1 \mod 3$ implies $\sqrt{v} \in 1+3\Z_3.$ 
\begin{lemma} \label{v}
Let $p =3$ and let $v\in C^{(3)}$. 
Let  $\mathcal{N}_v$ be the subspace of $ H^1(\Gamma_v, \operatorname{Ad}^0(\overline{\rho}))$ that preserves $\mathcal{C}_v$.
The pair $(\mathcal{N}_v, \mathcal{C}_v)$ is balanced in the sense that  
$$
\dim \mathcal{N}_v = \dim H^0(\Gamma_v, \operatorname{Ad}^0(\overline{\rho}) )
$$
and is spanned by the local cohomology class
$$
r_v: 
\sigma_v \mapsto 
\begin{pmatrix}
0 & 0 \\
0 & 0
\end{pmatrix},
\quad 
 \tau_v \mapsto 
\begin{pmatrix}
0 & 1 \\
0 & 0
\end{pmatrix}.
$$
\end{lemma}
\begin{proof}
We proceeds as for Lemma \ref{vnice}, first observing that $r_v$ preserves deformations of the form given in \eqref{good}, then complete the proof by appealing to Lemma \ref{localranks3}.
\end{proof}

\begin{lemma} \label{remainbalanced3}
    Suppose 
    $$
    \overline{\rho}: \Gamma_S \to GL(2, \F_{3^f})
    $$
    is irreducible with image containing 
    $\operatorname{SL}(2, \F_{3^f})$ and assume the global setting is balanced.
    If we allow ramification at an auxiliary prime $v\in C^{(3)}$, then the global setting remains balanced in the sense that
    $$
\dim H^1_{\mathcal{N}} (\Gamma_{S\cup\{v\}}, \operatorname{Ad}^0(\overline{\rho}))
=\dim H^1_{\mathcal{N}^\perp} (\Gamma_{S\cup\{v\}}, \operatorname{Ad}^0(\overline{\rho})^*).$$
\end{lemma}
\begin{proof}
    This follows from Lemma \ref{v} and Wiles' formula. 
\end{proof}

\subsection{Existence of $p$-adic representations without annihilating dual Selmer}
Theorem \ref{rankone} below focuses on a global setting with a nontrivial dual Selmer group of rank one; in particular, the technique in \cite{artinII, SFM} based on an effective construction of infinitesimal deformations is not available. 
The result we prove is that either the relevant global deformation ring is flat to begin with, so that a $p$-adic lift is known to exist; otherwise, we can make the global ring flat by allowing ramification at one additional prime, still keeping the global setting balanced of rank one. 
In particular, we prove that a $p$-adic lift exists \emph{without} annihilating the dual Selmer group. 

\begin{theorem} \label{rankone}
    Let $p \geq 3$ be a prime and let  $S$ be a finite set of places of $\Q$ containing $p$ and the archimedian place. Consider an irreducible residual Galois representation 
    $$
\overline{\rho}: \Gamma_S \to \operatorname{GL}(2, \F_{p^f}),
    $$
    with image containing $\operatorname{SL}(2, \F_{p^f})$
    whose parity is allowed to be either even or odd.
    Assume the global setting is balanced of rank one (cf. Definition \ref{balancedrankn}), and let $g$ span the Selmer group and $\gamma$ span the dual Selmer group. 
    Exactly one of the following is true:
    \begin{enumerate}
        \item $R_{(\mathcal{N}_v)_{v \in S}}$ 
        has a characteristic zero point: for some finite extension $\mathcal{O}/\Z_p$ there is an irreducible representation
        \[
        \rho: \Gamma_S \to GL(2,\mathcal{O})
        \]
    such that $\rho \equiv \overline{\rho} \bmod p$ and $\rho|_{\Gamma_v} \in \mathcal{C}_v$ for all $v\in S$.
        \item $R_{(\mathcal{N}_v)_{v \in S}}$ does not have a characteristic zero point; then there is a Chebotarev set $C^{(p)}$ such that for any prime $w\in C^{(p)}$ 
there exists, for some finite extension $\mathcal{O}/\Z_p$,
an irreducible representation
$$
\rho: \Gamma_{S\cup \{w\}} \to GL(2, \mathcal{O})
$$
ramified at $w$
with 
$\rho \equiv \overline{\rho} \mod p$
and
$\rho|_v \in \mathcal{C}_v$ for all 
$v\in S\cup \{ w\}.$ 
    \end{enumerate}
\end{theorem}

\begin{proof}
By Lemma \ref{1gen1rel} and the Weierstrass preparation theorem,
$$
R_{(\mathcal{N}_v)_{v \in S}} \simeq \W[[T]]/(p^\mu h(T))
$$
for some $\mu \geq 0$ and $h(T) \in \W[T]$ a distinguished polynomial, i.e. $h(T)$ is monic with all non-leading coefficients in the maximal ideal of $\W$. 
If $R_{(\mathcal{N}_v)_{v \in S}}$ has a characteristic zero point, 
the statement is clear. From now on, let us assume that 
$R_{(\mathcal{N}_v)_{v \in S}}$ does not have a characteristic zero point. Then we must have $\mu \geq 1$ and, moreover, it suffices to assume that 
$$
R_{(\mathcal{N}_v)_{v \in S}} \simeq
(\W/p^{\mu}\W) [[T]].
$$
Let 
$$\rho_\mathcal{N}: \Gamma_S \to GL(2,R_{(\mathcal{N}_v)_{v \in S}})$$
be the versal deformation of $\overline{\rho}$.
Let $d \in \Z_{\geq 2}$.
Firstly, suppose $p\geq 5$. Let $C_d^{(p)}$ be the Chebotarev set of primes $w$ such that
\begin{enumerate}
\item $w$ is nice cf. Definition \ref{nice},
\item $\rho_\mathcal{N}(\sigma_w) 
=\begin{pmatrix}
w(1+T^d) & 0 \\
0 & (1+T^d)^{-1} 
\end{pmatrix} \mod T^{d+1},$ \label{i}
\item $g|_{\Gamma_w} = 0,$ \label{ii}
\item $\gamma|_{\Gamma_w} \neq 0.$ \label{iii}
\end{enumerate}

If $p=3,$
let $C_d^{(3)}$ be the Chebotarev set of primes $w$ such that
\begin{enumerate}
\item $w \in C^{(3)}$ (cf. Definition \ref{C3}),
\item $\rho_\mathcal{N}(\sigma_w) 
=\begin{pmatrix}
\sqrt{w}(1+T^d) & 1 \\
0 & \sqrt{w}^{-1}(1+T^d)^{-1} 
\end{pmatrix} \mod T^{d+1},$ \label{i}
\item $g|_{\Gamma_w} = 0,$ \label{ii}
\item $\gamma|_{\Gamma_w} \neq 0.$ \label{iii}
\end{enumerate}
We will check that $C_d^{(p)}$ is non-empty for all $p\geq 3$.
Define $\Gamma_K :=\ker \overline{\rho}$ and let 
Let 
$K=\Q(\overline{\rho})$ be the field fixed by 
$\ker \overline{\rho}$ and define
$\Gamma_K := \ker \overline{\rho} = \Gal(\overline{K}/K)$.
Let $K_g/K$ be the extension of $K$ fixed by $\ker (g|_{\Gamma_K})$ and let
$K_\gamma/K$ be the field fixed by $\ker (\gamma|_{\Gamma_K})$.
We observe that
$\operatorname{Ad}^0(\overline{\rho})$ and $\operatorname{Ad}^0(\overline{\rho})^*$ are irreducible with 
\[
\Gal(K_g/K) \simeq \Ad, \quad \Gal(K_\gamma/K)\simeq \Ad^*,
\]
yet 
\[
\operatorname{Ad}^0(\overline{\rho}) \not\simeq \operatorname{Ad}^0(\overline{\rho})^*.\]
Hence $K_g/K$ and
$K_\gamma/K$ are linearly disjoint over $K$.
Since $d > 1$, the condition defining $\rho_\mathcal{N}(\sigma_w) \mod T^{d+1}$ is independent of the conditions
on $g$ and $\gamma$. Finally, for both $p\geq 5$ and for $p=3$, the first condition only involves the residual field $K$, whereas the remaining conditions involve fields above $K$. Hence, $C_d^{(p)}$ is an infinite set of primes for any $p\geq 3.$

Let $w\in C_d^{(p)}$. For any $p\geq 3$ it follows by
Lemma \ref{remainbalanced5} and Lemma \ref{remainbalanced3} that the global setting remains balanced after allowing ramification at $w$, i.e.
$$
\dim H^1_{\mathcal{N}} (\Gamma_{S\cup\{w\}}, \operatorname{Ad}^0(\overline{\rho}))
=\dim H^1_{\mathcal{N}^\perp} (\Gamma_{S\cup\{w\}}, \operatorname{Ad}^0(\overline{\rho})^*).
$$
In addition, the global setting remains of rank one, which we will now prove. 
By assumption, the generator $g$ of $H^1_{\mathcal{N}} (\Gamma_{S}, \operatorname{Ad}^0(\overline{\rho}))$ has 
$g|_{G_w} = 0,$ so $$g\in H^1_{\mathcal{N}} (\Gamma_{S\cup\{w\}}, \operatorname{Ad}^0(\overline{\rho})).$$ 
Therefore,
$$
1 \leq \dim H^1_{\mathcal{N}} (\Gamma_{S\cup\{w\}}, \operatorname{Ad}^0(\overline{\rho})).
$$
By an application of Wiles' formula, the injective inflation map
$$
H^1(\Gamma_{S}, \operatorname{Ad}^0(\overline{\rho})) \to H^1(\Gamma_{S \cup \{w \}}, \operatorname{Ad}^0(\overline{\rho}) )
$$
has cokernel of dimension less than or equal to $\dim H^2(\Gamma_w, \operatorname{Ad}^0(\overline{\rho})) = 1$.
The last equality uses Lemma \ref{localranks5} for $p\geq 5$ and Lemma \ref{localranks3} for $p=3$.
In particular, 
$$
\dim H^1_{\mathcal{N}} (\Gamma_{S\cup\{w\}}, \operatorname{Ad}^0(\overline{\rho}))
\leq \dim H^1_{\mathcal{N}} (\Gamma_{S}, \operatorname{Ad}^0(\overline{\rho})) + 1 = 2.
$$
So the global setting is balanced of either rank one or two. 

If the global setting was balanced of rank two, 
then $H^1_{\mathcal{N}^\perp} (\Gamma_{S\cup\{w\}}, \operatorname{Ad}^0(\overline{\rho})^*)$ would be the span 
of two linearly independent classes $\phi_1^{(w)},\phi_2^{(w)}$.
At the same time, the injective inflation map
$$
 H^1(\Gamma_{S}, \operatorname{Ad}^0(\overline{\rho})^*) \to H^1(\Gamma_{S \cup \{w \}}, \operatorname{Ad}^0(\overline{\rho})^* )
$$
has cokernel of rank less than or equal to $\dim H^2(\Gamma_w, \operatorname{Ad}^0(\overline{\rho})^*) = 1$.
This means that there exist $\lambda_1, \lambda_2 \in \F_{p^f}$, not both equal to zero, such that $\lambda_1 \phi_1^{(w)} + \lambda_2\phi_2^{(w)}$ inflates from 
$H^1(\Gamma_{S}, \operatorname{Ad}^0(\overline{\rho})^*).$
Let $\phi^{(w)}_{12}=\lambda_1 \phi_1^{(w)} + \lambda_2\phi_2^{(w)}.$
At inertia $\tau_w \in \Gamma_w,$
$$
\phi_{12}^{(w)} (\tau_w) = 0.
$$
Since $\phi_{12}^{(w)}|_{\Gamma_v} \in \mathcal{N}_v$ for all $v\in S,$
$$\phi_{12}^{(w)} \in 
H^1_\mathcal{N^\perp}(\Gamma_{S}, \operatorname{Ad}^0(\overline{\rho})^*) = \operatorname{span} \{ \gamma \},$$
so there is some $c\in \F_{p^f}^\times$ such that 
$$
\phi_{12}^{(w)} = c \gamma.
$$
By assumption, $\gamma|_{\Gamma_w} \neq 0,$ so it follows that $\phi_{12}^{(w)}|_{\Gamma_w} \neq 0.$
On the other hand, since $\phi_{12}^{(w)}|_{\Gamma_w} \in \mathcal{N}_w,$
there is some $c_1\in \F_{p^f}$ such that 
$$
\phi_{12}^{(w)}|_{\Gamma_w} = c_1\cdot r_w.   
$$
By Lemma \ref{vnice} (if $p\geq 5$) or Lemma \ref{v} (if $p=3$), 
$$
\phi_{12}^{(w)}(\tau_w) = c_1 r_w(\tau_w)
=
c_1 
\begin{pmatrix}
    0 & 1 \\
    0 & 0
\end{pmatrix}.
$$
It follows that $c_1 =0.$ But then $\phi_{12}^{(w)}|_{\Gamma_w} = 0,$ a contradiction.
In conclusion, after allowing ramification at $w$, the global setting remains of balanced of rank one. 
By the Weierstrass preparation theorem, 
$$
R_{(\mathcal{N}_v)_{v \in S\cup \{w\} }} \simeq \W[[T]]/(p^{\mu_w}h^{(w)}(T))
$$
where $h^{(w)}(T)$ is a distinguished polynomial. Suppose $\mu_w \geq 1$. There is a surjective morphism
\[
\begin{tikzcd} \pi^{(w)}_{d+1}: 
R_{(\mathcal{N}_v)_{v \in S\cup \{w\} }} \arrow[r, tail, twoheadrightarrow] & 
\F_{p^f}[[T]]  \arrow[r, tail, twoheadrightarrow]  & 
\F_{p^f}[[T]]/ (T^{d+1}),
\end{tikzcd}
\]
and since $\mu \geq 1$, there is a surjective morphism
\[
\begin{tikzcd} \pi_{d+1}:
R_{(\mathcal{N}_v)_{v \in S }} \arrow[r, tail, twoheadrightarrow] & 
\F_{p^f}[[T]]  \arrow[r, tail, twoheadrightarrow]  & 
\F_{p^f}[[T]]/ (T^{d+1}).
\end{tikzcd}
\]
The deformations 
$\pi^{(w)}_{d+1} \circ \rho_\mathcal{N}^{(w)}$ and 
$\pi^{}_{d+1} \circ \rho_\mathcal{N}^{}$ to 
$\operatorname{GL}(2,\F_{p^f}[[T]]/ (T^{d+1}))$ have the property that 
$$
\pi^{(w)}_{d+1} \circ \rho_\mathcal{N}^{(w)}
\equiv 
\pi^{}_{d+1} \circ \rho_\mathcal{N}^{}
\mod T^d.
$$
Since $H^1 (\Gamma_{S\cup\{w\}}, \operatorname{Ad}^0(\overline{\rho}))$
acts transitively on deformations to 
$\operatorname{GL}(2,\F_{p^f}[[T]]/ (T^{d+1}))$  that agree mod $T^d$, there is some 
$h \in H^1 (\Gamma_{S\cup\{w\}}, \operatorname{Ad}^0(\overline{\rho}))$
such that 
$$\pi^{(w)}_{d+1} \circ \rho_\mathcal{N}^{(w)}
= (I+T^d h ) (\pi^{}_{d+1} \circ \rho_\mathcal{N}^{}).
$$ 
Clearly, $h|_{\Gamma_v} \in \mathcal{N}_v$ for all $v\in S$. 
At the same time, 
$\rho_\mathcal{N}|_{\Gamma_w} \notin \mathcal{C}_w$, 
and hence 
$\pi^{}_{d+1} \circ \rho_\mathcal{N}^{}|_{\Gamma_w} \notin \mathcal{C}_w$. Thus $h|_{\Gamma_w} \notin \mathcal{N}_w$.
Consider the morphisms
$$
\varphi_{w}: H^1(\Gamma_{S\cup \{w \}}, \operatorname{Ad}^0(\overline{\rho})) \to \bigoplus_{v\in S} H^1(\Gamma_v,\operatorname{Ad}^0(\overline{\rho}))/\mathcal{N}_v,
$$
$$
\varphi_{S}: H^1(\Gamma_{S}, \operatorname{Ad}^0(\overline{\rho})) \to \bigoplus_{v\in S} H^1(\Gamma_v, \operatorname{Ad}^0(\overline{\rho}))/\mathcal{N}_v,
$$
and the induced maps on dual groups,
$$
\varphi^*_{w}: H^1(\Gamma_{S\cup \{w \}}, \operatorname{Ad}^0(\overline{\rho})^*) \to 
\bigoplus_{v\in S} H^1(\Gamma_v, \operatorname{Ad}^0(\overline{\rho})^*)/\mathcal{N}^\perp_v
\oplus H^1(\Gamma_w, \operatorname{Ad}^0(\overline{\rho})^*)/\{ 0 \},
$$
$$
\varphi^*_{S}: H^1(\Gamma_{S}, \operatorname{Ad}^0(\overline{\rho})^*) \to \bigoplus_{v\in S} H^1(\Gamma_v, \operatorname{Ad}^0(\overline{\rho}))/\mathcal{N}^\perp_v,
$$
Suppose $h^* \in \ker \varphi_w^*.$
    Since $h^*|_{\Gamma_w} =0$, $h^*$ inflates from 
$H^1(\Gamma_{S}, \operatorname{Ad}^0(\overline{\rho})^*)$. It follows that
$$
h^* = c \cdot \gamma 
$$
for some $c\in \F_{p^f}$. 
By assumption, $\gamma|_{\Gamma_w} \neq 0.$ It follows that 
$c=0,$ so $h^*=0$. That is, 
$$
\dim\ker \varphi_w^*=0.
$$
By two applications of Wiles' formula,
\begin{IEEEeqnarray*}{rCl}
  & &   (\dim \ker \varphi_{w} - \dim \ker \varphi_w^*)
-
(\dim \ker \varphi_S - \dim \ker \varphi_S^*) \\
 &=& (\dim \ker \varphi_{w} - 0) - (1-1) \\
 &=& 
 \dim H^1(\Gamma_w, \operatorname{Ad}^0(\overline{\rho})) - 
 \dim H^0(\Gamma_w, \operatorname{Ad}^0(\overline{\rho})) \\
  &=& 2 - 1 =1.
\end{IEEEeqnarray*}
We conclude that $\dim \ker \varphi_{w} =1$.
Since 
$$
h, g \in \ker \varphi_{w},
$$
it follows that $$g =c_0 \cdot h$$ 
for some $c_0 \in \F_{p^f}$.
However, $g|_{\Gamma_w} = 0$ while $h|_{\Gamma_w} \neq 0$. This forces
$c_0=0.$ But then $g=0,$ a contradiction. We conclude that 
$$
\mu_w= 0.
$$
\end{proof}

\begin{corollary}[Controlling the rank of the mod $p$ ring]
    Let $p\geq 3$ and suppose we are in the setting of Theorem \ref{rankone}.
    For any $d>1,$ there exists a Chebotarev set $C_d^{(p)}$ of auxiliary primes 
such that for $w \in C_d^{(p)}$,
$$
R_{(\mathcal{N}_v)_{v \in S\cup \{w\} }} 
\simeq \W[[T]]/(T^d + a_{d-1}T^{d-1} + \cdots + a_0)
$$
for some $a_i \in p\W,$ $i=0,\ldots, d-1$. Hence
$$
\dim_{\F_{p^f}} (R_{(\mathcal{N}_v)_{v \in S\cup \{w\} }}/pR_{(\mathcal{N}_v)_{v \in S\cup \{w\} }}) = d.
$$
\end{corollary}
\begin{proof}
Let us keep the notation of the proof of Theorem \ref{rankone}. We will show that $h^{(w)}(T)$ has degree $d$. 
Since 
$$
\rho_\mathcal{N} \equiv \rho_\mathcal{N}^{(w)} \mod (p,T^d),
$$
it follows that 
\[
\begin{tikzcd} 
R_{(\mathcal{N}_v)_{v \in S\cup \{w\} }}  \arrow[r, tail, twoheadrightarrow]  & 
\F_{p^f}[[T]]/ (T^{d}).
\end{tikzcd}
\]
This implies that 
$$
(h^{(w)}(T) \mod p) \in (T^d).
$$
Hence the degree of $h^{(w)}(T)$ is at least $d$.
If the degree is at least $d+1$, then 
$$
h^{(w)}(T) \equiv T^{d+1+k}\mod p.
$$
for some $k\geq 0.$
Hence
\[
\begin{tikzcd} 
R_{(\mathcal{N}_v)_{v \in S\cup \{w\} }}  \arrow[r, tail, twoheadrightarrow]  & 
\F_{p^f}[[T]]/ (T^{d+1}),
\end{tikzcd}
\]
and we get the same contradiction as above. We conclude that the degree is exactly $d$.
\end{proof}

\begin{remark}
    Note that the proof above also works for odd $\overline{\rho}$ with a geometric condition at $p$.
\end{remark}

\section{Explicit examples of even deformation rings that are flat over $\Z_3$}

In \cite{auto},
an explicit family of even, irreducible representations
$$
\overline{\rho}^{(\ell)}: \Gal(\overline{\Q}/\Q ) \to \operatorname{SL}(2,\F_3)
$$
ramified only at primes $\ell \equiv 1 \bmod 3$ was constructed. 
For such primes $\ell$, there is a totally real cubic subfield $L$ of $\Q(\zeta_\ell)$. If the class number $h_L$ of $L$ is divisible by 4, there is an $A_4$-extension $K/\Q$ of discriminant $\ell^2$ containing $L$. Since $A_4\simeq \operatorname{PSL}(2,\F_3)$, the fields $K$ were shown to give rise to a family of even Galois representations $ \overline{\rho} =\overline{\rho}^{(\ell)}$ as above. 

\begin{lemma} \label{P}
Let $\ell=7489$ and let $L$ be the cubic subfield of $\Q(\zeta_\ell)$. 
Then the class number of $L$ is $h_L = 28$, 
and there exists a unique
$A_4$-extension $K/\Q$ containing $L$ and a corresponding even representation 
$$\overline{\rho} = \overline{\rho}^{(\ell)}: \Gal(\overline{\Q}/\Q) \to \operatorname{SL}(2,\F_3).$$  
Let $T=\{3,\ell\}$ and 
let $K_{T}^{(3)}$ denote the maximal 3-elementary abelian extension of $K$ 
unramified outside the places in $K$ above the places in $T$.
As $\Gal(K/\mathbb{Q})$-modules, there is a decomposition
$$
\Gal(K^{(3)}_T/K) \simeq \operatorname{Ad}^0(\overline{\rho})
\oplus \operatorname{Ad}^0 (\overline{\rho})
\oplus \mathbb{F}_3
\oplus \mathbb{F}_3
\oplus \mathbb{F}_3.
$$
Moreover, 
$$
\dim H^1 (\Gal(\Q_T/\Q), \operatorname{Ad}^0\overline{\rho}) =2.
$$
\end{lemma}
\begin{proof}
The prime $\ell =7489$ was found by numerical computations in magma and pari-gp. 
A ray class group computation in magma shows that $\Gal(K^{(3)}_T/K) $ has  9 generators. Each trivial representation occurring in the Jordan-H\"{o}lder sequence for $\Gal(K^{(3)}_T/K) $ descends to $L$. A ray class group computation over $L$ shows that there are 3 such trivial representations. The result follows as in Lemma 4 of \cite{even2}. The dimension of $H^1 (\Gal(\Q_T/\Q), \operatorname{Ad}^0\overline{\rho})$ equals 
 the number of adjoints in the decomposition of $\Gal(K^{(3)}_T/K)$ as an $A_4$-module.
\end{proof}

\begin{proposition} \label{Selmerone}
Let $\ell=7489$ and let $S=\{\infty,3,\ell\}$. Then
$$
\dim H^1_\mathcal{N} (\Gamma_S, \operatorname{Ad}^0\overline{\rho}) =1.
$$
\end{proposition}
\begin{proof} Let $T = \{3, \ell\}$.
Clearly,
\[
\begin{tikzcd} 
H^1_\mathcal{N} (\Gamma_S, \operatorname{Ad}^0\overline{\rho}) \arrow[r, hook]  & 
H^1(\Gamma_T, \operatorname{Ad}^0\overline{\rho}).
\end{tikzcd}
\]
By Lemma \ref{P}, it follows that
$$
\dim H^1_\mathcal{N} (\Gamma_S, \operatorname{Ad}^0\overline{\rho}) \leq 2.
$$
Suppose, to reach a contradiction, that
 $
\dim H^1_\mathcal{N} (\Gamma_S, \operatorname{Ad}^0\overline{\rho}) =2
$
so that 
$$
 H^1_\mathcal{N} (\Gamma_S, \operatorname{Ad}^0\overline{\rho})
 \simeq
H^1(\Gamma_T, \operatorname{Ad}^0\overline{\rho}).
$$
As we chose $\mathcal{N}_\ell = H^1_{\operatorname{unr}}(\Gamma_\ell, \Ad)$, observe that any basis for $H^1_\mathcal{N} (\Gamma_S, \operatorname{Ad}^0\overline{\rho})$ consists of two linearly independent cohomology classes $f,g$ which are unramified at $\ell$. 
Let $\Gamma_K = \ker \overline{\rho}$ 
and let $K_f/K$ and $K_g/K$ be the fields fixed by $\ker(f|_{\Gamma_K})$
and $\ker(g|_{\Gamma_K})$, respectively. 
Due to the contributions from $\Gal(K_f/K)$ and $\Gal(K_g/K)$, 
the Galois group of the 3-Frattini extension of $K$ unramified outside \{3\} would have at least $6$ generators as a vector space over $\F_3$.
However a ray class computation shows that it only has 5 generators. 
\end{proof}

\begin{corollary}
Let $\ell=7489$, and let $\mu \geq 0$ be such that
$$
R_{(\mathcal{N}_v)_{v \in S}} \simeq 
\Z_3[[T]]/(3^{\mu}h(T)).
$$
\begin{enumerate}
\item If $\mu=0,$ there exists a finite extension $\mathcal{O}/\Z_3$ of rings and
an even Galois representation
$$
\rho: \Gamma_{S} \to SL(2, \mathcal{O})
$$
lifting $\overline{\rho}$ such that $\rho|_{\Gamma_v} \in \mathcal{C}_v$ for all $v\in S$.
\item If $\mu\geq 1,$ and if there is no characteristic zero lift at the minimal level $S$, then there exist infinitely many auxiliary primes $w$ such that the even deformation ring
$$
R_{(\mathcal{N}_v)_{v \in S\cup\{w\}}} \simeq 
\Z_3[[T]]/(h^{(w)}(T))
$$
is finite and flat over $\Z_3$, and  
there exists a finite extension $\mathcal{O}/\Z_3$ of rings and
an even Galois representation
$$
\rho^{(w)}: \Gamma_{S\cup\{ w\}} \to \operatorname{SL}(2, \mathcal{O})
$$
such that  $\rho \equiv \overline{\rho} \mod 3,$ $\rho$ is ramified at $w$, and  $\rho^{(w)}|_{\Gamma_v} \in \mathcal{C}_v$ for all $v\in S\cup\{w\}$.
\end{enumerate}
\end{corollary}
\begin{proof}
This follows from Proposition \ref{Selmerone} and Theorem \ref{rankone}.
\end{proof}

In \cite{classification}, conditions on $\ell$ were given for members of the family
$\{ \overline{\rho}^{(\ell)} \}$
to admit lifts $\rho^{(\ell)}$ onto $\operatorname{SL}(2,\Z_3)$. 
The smallest such $\ell$ are 
$\ell = 163, 277, 349, 607,\ldots$
Fix such an $\ell$. For $S=\{3,\ell,\infty\}$, as shown in \cite{classification}, 
$$
H^1_{\mathcal{N}}(\Gal(\Q_S/\Q), \Ad) =0.
$$
In \cite{level}, a condition on primes  $w$ was given such that
$$H^1_{\mathcal{N}}(\Gal(\Q_{S \cup \{w\} }/\Q), \Ad) = 0$$
and such that there exists a \emph{unique} even surjective representation 
$$
\rho^{(\ell, w)}: \Gal(\Q_{S \cup \{w\} }/\Q) \to \operatorname{SL}(2,\Z_3)
$$
with $\rho^{(\ell, w)} \equiv \overline{\rho}^{(\ell)} \bmod 3$
and $\rho^{(\ell,w)} \in \mathcal{C}_v$ for all $v\in S\cup \{w\}$.
On the other hand, if $w$ fails the condition alluded to above, 
then the global setting becomes balanced of rank one: 
$$\dim H^1_{\mathcal{N}}(\Gal(\Q_{S \cup \{w\} }/\Q), \Ad) 
= \dim H^1_{\mathcal{N}^\perp}(\Gal(\Q_{S \cup \{w\} }/\Q), \Ad^*) = 1. $$
For instance, if $\ell=349,$ then the following primes give rise to a balanced global setting of rank one:
$$
w = 19, 193, 271, 331, 367, 373, \ldots
$$
In this case, let $\mu \geq 0$ be such that
$$
R_{(\mathcal{N}_v)_{v \in S \cup \{w\} }} \simeq 
\Z_3[[T]]/(3^{\mu}h^{(w)}(T))
$$
where $h^{(w)}(T)$ is distinguished. 
If $\mu=0,$ there exists a finite extension $\mathcal{O}/\Z_3$ of rings and an even representation
$$
\rho^{}: \Gamma_{S \cup \{ w\} }  \to \operatorname{SL}(2, \mathcal{O})
$$
lifting $\overline{\rho}$ such that $\rho|_{\Gamma_v} \in \mathcal{C}_v$ for all $v\in S \cup \{ w\}$. If $\mu\geq 1,$ and if there is no characteristic zero lift at the minimal level $S$, then there exist infinitely many auxiliary primes $q$ such that, after allowing ramification at $q$, the even deformation ring
$$
R_{(\mathcal{N}_v)_{v \in S\cup\{w, q\}}} \simeq 
\Z_3[[T]]/(h^{(w,q)}(T))
$$
is finite and flat over $\Z_3$ for some distinguished polynomial $h^{(w,q)}(T)$,
and  
there exists a finite extension $\mathcal{O}/\Z_3$ of rings and
an even representation
$$
\rho^{(q)}: \Gamma_{S\cup\{ w,q\}} \to \operatorname{SL}(2, \mathcal{O})
$$
ramified at $q$ with  
$\rho^{(q)} \equiv \overline{\rho} \mod 3$ and $\rho^{(q)}|_{\Gamma_v} \in \mathcal{C}_v$ for all $v\in S\cup\{w,q\}$.

\section{Leopoldt's conjecture and flatness of even deformation rings}

\begin{definition}
    Let $K$ be a number field. If $L$ is a normal extension of $K$ such that $\Gal(L/K) \simeq \Z_p,$ we say that $L/K$ is a $\Z_p$-extension. 
\end{definition}

\begin{remark}[Leopoldt's conjecture] Let $K$ be a number field with $r_1$ real embeddings and $r_2$ pairs of complex embeddings and let $p$ be a prime. 
Let $S$  be the set of primes above $p$ in $K$, and let $K_S^{\operatorname{ab}}$ be the maximal abelian extension of $K$ unramified outside $S$.
Leopoldt's conjecture for $K$ is that 
$$
\dim_{\Q_p} \Gal(K^{\operatorname{ab}}_S/K) \otimes \Q = 1 + r_2, 
$$
so the number of independent $\Z_p$-extensions of $K$ is equal to $1+r_2$.
If $K$ is a field fixed by the kernel of a $2$-dimensional even representation, then $K$ is totally real so Leopoldt's conjecture predicts that there is only one $\Z_p$-extension of $K$ unramified away from $p$.
\end{remark}

\begin{theorem}
Consider the irreducible even representation  
$$\overline{\rho}^{(\ell)}: \Gal(\overline{\Q}/\Q)\to \operatorname{SL}(2,\F_3)$$
unramified outside $S =\{3, \ell, \infty\}$.  
Suppose the global even deformation ring has one generator and one relation: 
$$R_\mathcal{N} = R_{(\mathcal{N}_v)_{v\in S}} \simeq \Z_3[[T]]/(f).$$
Let $\rho_1$ be an even representation onto $\operatorname{SL}(2,\Z/3^2\Z)$
lifting $\overline{\rho}$ and let $K_1$ be the field fixed by the kernel of $\rho_1$.
If Leopoldt's conjecture is true for the field $K_1$, then the global even deformation ring $R_\mathcal{N}$ is flat over $\Z_3$.    
\end{theorem}
\begin{proof}
Let $p=3$. 
If $p$ divides the relation $f,$ then 
the global mod $p$ ring $R_\mathcal{N}/pR_\mathcal{N}$ is a power series ring over $\F_p$ in one variable. 
Let  
$\varrho$ denote the universal deformation for $R_\mathcal{N}/pR_\mathcal{N} \simeq \F_p[[T]]$ 
and let $K_n$ be the field fixed by 
the kernel of $(\varrho \bmod T^{n+1})$ for all $n\geq 1$. 
The fields $K_n$ form an infinite tower of number fields
    $$\cdots K_n/K_{n-1}/\cdots K_2/K_1/K_0 =\Q(\overline{\rho})$$ 
    such that 
    $\Gal(K_n/K_{n-1}) \simeq \operatorname{Ad}^0(\overline{\rho})$ as $\F_p[\Gal(\Q(\overline{\rho})/\Q)]$-modules 
    for all $n\geq 1$. By definition, the condition $\varrho|_{\Gamma_\ell} \in \mathcal{C}_\ell$ implies that 
    the ramification of $\varrho$ at $\ell$ 
    must be exhausted already in $K_1$, leaving $p$ as the only ramified prime in the tower above $K_1$. 
Thus there are at least three independent $\Z_p$-extensions above $K_1$ unramified away from $p$. 
Since $K_1/\Q$ is totally real, this violates Leopoldt's conjecture for $K_1$. 
\end{proof}

\bibliographystyle{amsplain} 
\bibliography{references.bib}

\nocite{*}

\end{document}